\documentclass[12pt]{amsart}
\usepackage{amsfonts}
\usepackage{amssymb}
\usepackage{amsmath}

\newtheorem{theorem}{Theorem}[section]

\newtheorem{proposition}{Proposition}[section]
\newtheorem{corollary}{Corollary}[section]

\newtheorem{definition}{Definition}[section]
\newtheorem{example}{Example}[section]
\numberwithin{equation}{section}
\def\({\left ( }
\def\){\right )}
\def\<{\left < }
\def\>{\right >}

\begin{document}
\title[Transversal Lightlike Submanifolds]{Transversal Lightlike
Submanifolds of Metallic Semi-Riemannian Manifolds}
\author{Feyza Esra Erdo\u{g}an}
\address{Faculty of Education, Department of Mathematics, Ad\i yaman
University, 02040 Ad\i yaman, TURKEY}
\email{ferdogan@adiyaman.edu.tr}

\begin{abstract}
The Metallic Ratio is fascinating topic that continually generated news
ideas. A Riemannian manifold endowed with a Metallic structure will be
called a Metallic Riemannian manifold. The main purpose of the present paper
is to study the geometry of transversal lightlike submanifolds and radical
transversal lightlike submanifolds and of Metallic Semi-Riemannian
manifolds. We investigate the geometry of distributions and obtain necessary
and sufficient conditions for the induced connection on these manifolds to
be metric connection. We also obtain characterization of transversal
lightlike submanifolds of Metallic semi-Riemannian manifolds. Finally, we
give an example.
\end{abstract}

\maketitle

\section{INTRODUCTION}

Lightlike submanifolds are one of the most interesting topics in
differential geometry. It is well known that a submanifold of a Riemannian
manifold is always a Riemannian one. Contrary to that case, in
semi-Riemannian manifolds the induced metric by the semi-Riemann metric on
the ambient manifold is not necessarily nondegenerate. Since the induced
metric is degenerate\ on lightlike submanifolds, the tools which are used to
investigate the geometry of submanifolds in Riemannian case are not
favorable in semi-Riemannian case and so the classical theory can not be
used to define any induced object on a lightlike submanifold. The main
difficulties arise from the fact that the intersection of the normal bundle
and the tangent bundle of a lightlike submanifold is nonzero. In 1996,
K.Duggal-A.Bejancu \cite{DB} put forward the general theory of lightlike
submanifolds of semi-Riemannian manifolds in their book. In order to resolve
the difficulties that arise during studying lightlike submanifolds, they
introduced a non-degenerate distribution called screen distribution to
construct a lightlike transversal vector bundle which does not intersect to
its lightlike tangent bundle. It is well-known that a suitable choice of
screen distribution gives rises to many substantial results in lightlike
geometry. Many authors have studied the geometry of lightlike submanifolds%
\cite{FC,BC,FBR,FBR1,NE,SBE,BIS,BSE1,BIS1} in different manifolds . For
further read we refer \cite{DB,DS} and the references therein.

Manifolds with various geometric structures are convenient to study
submanifold theory. In recent years, one of the most studied manifold types
are Riemannian manifolds with metallic structures. Metallic structures on
Riemannian manifolds allow many geometric results to be given on a
submanifold.

As a generalization of the Golden mean which contains the Silver mean, the
Bronze mean, the Copper mean and the Nickel mean etc., Metallic means family
was introduced by V. W. de Spinadel \cite{V.W4} in 2002. The positive
solution of the equation given by 
\begin{equation*}
x^{2}-px-q=0,
\end{equation*}%
for some positive integer $p$ and $q$, is called a $(p,q)$-metallic number 
\cite{V.W1,V.W3} which has the form%
\begin{equation*}
\sigma _{p,q}=\frac{p+\sqrt{p^{2}+4q}}{2}.
\end{equation*}%
For $p=q=1$ and $p=2,$ $q=1$, it is well-known that we have the Golden mean $%
\phi =\frac{1+\sqrt{5}}{2}$ and Silver mean $\sigma _{2,1}=1+\sqrt{2}$,
respectively$.$ The metallic mean family plays an important role to
establish a relationship between mathematics and architecture. For example
golden mean and silver mean can be seen in the sacred art of Egypt, Turkey,
India, China and other ancient civilizations \cite{V.W5}.

S. I. Goldberg, K. Yano and N. C. Petridis in (\cite{GY} and \cite{GY1})
introduced polynomial structures on manifolds. As some particular cases of
polynomial structures C. E. Hretcanu and M. Crasmareanu defined Golden
structure \cite{GH,GH2,GH3,CM} and some generalizations of this, called
metallic structure \cite{GCS}. Being inspired by the Metallic mean, the
notion of Metallic manifold $\breve{N}$ was defined in \cite{GCS} by a $%
(1,1) $-tensor field $\breve{J}$ on $\breve{N},$ which satisfies $\breve{J}%
^{2}=p\breve{J}+qI$, where $I$ is the identity operator on the Lie algebra $%
\chi (\breve{N})$ of vector fields on $\breve{N}$ and $p$, $q$ are fixed
positive integer numbers. Moreover, if $(\breve{N},g)$ is a Riemannian
manifold endowed with a metallic structure $\breve{J}$ such that the
Riemannian metric $\breve{g}$ is $\breve{J}$-compatible, i.e., $\breve{g}(%
\breve{J}V,W)=\breve{g}(V,\breve{J}W),$ for any $V,W\in \chi (\breve{N})$,
then $(\breve{g},\breve{J})$ is called metallic Riemannian structure and$(%
\breve{N},\breve{g},\breve{J})$ is a Metallic Riemannian manifold. Metallic
structure on the ambient Riemannian manifold provides important geometrical
results on the submanifolds, since it is an important tool while
investigating the geometry of submanifolds. Invariant, anti-invariant,
semi-invariant, slant and semi-slant submanifolds of a Metallic Riemannian
manifold are studied in \cite{MH,MH1,MH2} and the authors obtained important
characterizations on submanifolds of Metallic Riemannian manifolds.

One of the most important subclasses of Metallic Riemannian manifolds is the
Golden Riemannian manifolds. Many authors have studied Golden Riemannian
manifolds and their submanifolds in recent years (see \cite%
{GH,GH2,GH3,BM,EC,M,FC3}). N. Poyraz \"{O}nen and E. Ya\c{s}ar \cite{NE}
initiated the study of lightlike geometry in Golden semi-Riemannian
manifolds, by investigating lightlike hypersurfaces of Golden
semi-Riemannian manifolds. B. E. Acet introduced lightlike hypersurfaces in Metallic semi-Riemannian manifolds \cite{B}.

Motivated by the studies on submanifolds of Metallic Riemannian manifolds
and lightlike submanifolds of semi-Riemannian manifolds, in the present
paper we introduce the transversal lightlike submanifolds of a Metallic
semi-Riemannian manifold.

\bigskip

Considering given brief above, in this paper, we introduce transversal
lightlike submanifolds of Metallic semi-Riemannian manifolds and studied
their differential geometry. The paper is organized as follows: In Section 2
is devoted to basic definitions needed for the rest of the paper. In Section
3 and Section 4 , we introduce a Metallic semi-Riemannian manifold along
with its subclasses, namely radical transversal and transversal lightlike
submanifolds and obtain some characterizations. We investigate the geometry
of distributions and find necessary and sufficient conditions for induced
connection to be a metric connection. Furthermore, we give an example.

\section{PRELIMINARIES}

A submanifold $\acute{N}^{m}$ immersed in a semi-Riemannian manifold $(%
\breve{N}^{m+k},\breve{g})$ is called a lightlike submanifold if it admits a
degenerate metric $g$\ induced from $\breve{g},$ whose radical distribution $%
Rad\,T\acute{N}$ is of rank $r$, where $1\leq r\leq m.$ Then $Rad\,T\acute{N}%
=T\acute{N}\cap T\acute{N}^{\perp }$ , where 
\begin{equation}
T\acute{N}^{\perp }=\cup _{x\in \acute{N}}\left\{ u\in T_{x}\breve{N}\mid 
\breve{g}\left( u,v\right) =0,\forall v\in T_{x}\acute{N}\right\} .
\end{equation}%
Let $S(T\acute{N})$ be a screen distribution which is a semi-Riemannian
complementary distribution of $Rad\,T\acute{N}$ in $T\acute{N}$ i.e., $T%
\acute{N}=Rad\,T\acute{N}\perp S(T\acute{N}).$

We consider a screen transversal vector bundle $S(T\acute{N}^{\bot }),$
which is a semi-Riemannian complementary vector bundle of $Rad\,T\acute{N}$
in $T\acute{N}^{\bot }.$ Since, for any local basis $\left\{ \xi
_{i}\right\} $ of $Rad\,T\acute{N}$, there exists a lightlike transversal
vector bundle $ltr(T\acute{N})$ locally spanned by $\left\{ N_{i}\right\} $ 
\cite{DB}. Let $tr(T\acute{N})$ be complementary( but not orthogonal) vector
bundle to $T\acute{N}$ in $T\breve{N}^{\perp }\mid _{\acute{N}}$. Then, we
have%
\begin{eqnarray*}
tr(T\acute{N}) &=&ltr\,T\acute{N}\bot S(T\acute{N}^{\bot }), \\
T\breve{N}\,|\,_{\acute{N}} &=&S(T\acute{N})\bot \lbrack Rad\,T\acute{N}%
\oplus ltr\,T\acute{N}]\perp S(T\acute{N}^{\bot }).
\end{eqnarray*}%
Although $S(T\acute{N})$ is not unique, it is canonically isomorphic to the
factor vector bundle $T\acute{N}/Rad\,T\acute{N}$ \cite{DB}.

The following result is important for this paper.

\begin{proposition}
The lightlike second fundamental forms of a lightlike submanifold $\acute{N}$
do not depend on $S(T\acute{N}),$ $S(T\acute{N}^{\perp })$ and $ltrT\acute{N}
$ \cite{DB}.
\end{proposition}

We say that a submanifold $(\acute{N},g,S(T\acute{N}),$ $S(T\acute{N}^{\perp
}))$ of $\breve{N}$ is

Case 1: \ r-lightlike if $r<\min \{m,k\};$

Case 2: \ Co-isotropic if $r=k<m;S(T\acute{N}^{\perp })=\{0\};$

Case 3: \ Isotropic if $r=m=k;$ $S(T\acute{N})=\{0\};$

Case 4: \ Totally lightlike if $r=k=m;$ $S(T\acute{N})=\{0\}=S(T\acute{N}%
^{\perp }).$

\bigskip

The Gauss and Weingarten equations are%
\begin{eqnarray}
\breve{\nabla}_{W}U &=&\nabla _{W}U+h\left( W,U\right) ,\quad \forall W,U\in
\Gamma (T\acute{N}),  \label{7} \\
\breve{\nabla}_{W}V &=&-A_{V}W+\nabla _{W}^{t}V,\quad \forall W\in \Gamma (T%
\acute{N}),V\in \Gamma (tr(T\acute{N})),  \label{8}
\end{eqnarray}%
where $\left\{ \nabla _{W}U,A_{V}W\right\} $ and $\left\{ h\left( W,U\right)
,\nabla _{W}^{t}V\right\} $ belong to $\Gamma (T\acute{N})$ and $\Gamma (tr(T%
\acute{N})),$ respectively. Here, $\nabla $ and $\nabla ^{t}$ denote linear
connections on $\acute{N}$ and the vector bundle $tr\,(T\acute{N})$,
respectively. Moreover, we have%
\begin{eqnarray}
\breve{\nabla}_{W}U &=&\nabla _{W}U+h^{\ell }\left( {\small W,U}\right)
+h^{s}\left( {\small W,U}\right) ,\quad \forall W,U\in \text{$\Gamma (T%
\acute{N})$},  \label{9} \\
\breve{\nabla}_{W}N &=&-A_{N}W+\nabla _{W}^{\ell }N+D^{s}\left( {\small W,N}%
\right) ,\quad N\in \Gamma (\text{$ltr\,T\acute{N})$},  \label{10} \\
\breve{\nabla}_{W}Z &=&-A_{Z}W+\nabla _{W}^{s}Z+D^{\ell }\left( {\small W,Z}%
\right) ,\quad Z\in \Gamma (\text{$S(T\acute{N}^{\bot }))$}.  \label{11}
\end{eqnarray}%
Denote the projection of $T\acute{N}$ on $S(T\acute{N})$ by $P.$ Then by
using (\ref{7}), (\ref{9})-(\ref{11}) and the fact that $\breve{\nabla}$
being a metric connection , we obtain%
\begin{eqnarray}
\breve{g}(h^{s}\left( {\small W,U}\right) ,Z)+\breve{g}({\small U,D}^{\ell
}\left( {\small W,Z}\right) ) &=&\breve{g}\left( {\small A}_{Z}{\small W,U}%
\right) ,  \label{12} \\
\breve{g}\left( {\small D}^{s}\left( {\small W,N}\right) {\small ,Z}\right)
&=&\breve{g}\left( {\small N,A}_{Z}{\small W}\right) .  \label{13}
\end{eqnarray}%
From the decomposition of the tangent bundle of a lightlike submanifold, we
have%
\begin{eqnarray}
\nabla _{W}PU &=&\nabla _{W}^{\ast }PU+h^{\ast }\left( W,PU\right) ,
\label{14} \\
\nabla _{W}\xi &=&-A_{\xi }^{\ast }W+\nabla _{W}^{\ast t}\xi ,  \label{15}
\end{eqnarray}%
for $W,U\in \Gamma (T\acute{N})$ and $\xi \in \Gamma (Rad\,T\acute{N}).$ By
using above equations, we obtain%
\begin{eqnarray}
g\left( h^{\ell }\left( W,PU\right) ,\xi \right) &=&g\left( A_{\xi }^{\ast
}W,PU\right) ,  \label{16} \\
g\left( h^{s}\left( W,PU\right) ,N\right) &=&g\left( A_{N}W,PU\right) ,
\label{17} \\
g\left( h^{\ell }\left( W,\xi \right) ,\xi \right) &=&0,\quad A_{\xi }^{\ast
}\xi =0.  \label{18}
\end{eqnarray}%
In general, the induced connection $\nabla $ on $\acute{N}$ is not a metric
connection. Since $\breve{\nabla}$ is a metric connection, by using (\ref{9}%
) we get%
\begin{equation}
\left( \nabla _{W}g\right) \left( U,V\right) {\small =\bar{g}}\left( h^{\ell
}\left( W,U\right) ,V\right) {\small +\bar{g}}\left( h^{\ell }\left(
W,V\right) ,U\right) .  \label{19}
\end{equation}%
However, we note that $\nabla ^{\ast }$ is a metric connection on $S(T\acute{%
N})$.

\bigskip Fix two positive integers $p$ and $q$. The positive solution of the
equation 
\begin{equation*}
x^{2}-px-q=0,
\end{equation*}%
is entitled member of metallic means family.(\cite{V.W1}-\cite{V.W5}). These
numbers, denoted by%
\begin{equation}
\sigma _{p,q}=\frac{p+\sqrt{p^{2}+4q}}{2},  \label{19a}
\end{equation}%
are called $(p,q)$-metallic numbers.

\begin{definition}
A polynomial structure on a manifold $\breve{N}$ is called a Metallic
structure if it is determined by an $(1,1)$- tensor field $\breve{J}$, which
satisfies \ 
\begin{equation}
\breve{J}^{2}=p\breve{J}+qI,  \label{20}
\end{equation}%
where $I$ is the identity map on $\breve{N}$ and $p,q$ are positive
integers. Also,if 
\begin{equation}
\breve{g}(\breve{J}W,U)=\breve{g}(W,\breve{J}U)  \label{21}
\end{equation}%
\newline
holds Then the semi-Riemannian metric $\breve{g}$ is called $\breve{J}$%
-compatible, for every $U,W\in \Gamma (T\breve{N}).$ In this case $(\breve{N}%
,\breve{g},\breve{J})$ is named a Metallic semi-Riemannian manifold. Also a
Metallic semi-Riemannian structure $\breve{J}$ is called a locally Metallic
structure if $\breve{J}$ is parallel with respect to the Levi-Civita
connection $\breve{\nabla},$ that is 
\begin{equation}
\breve{\nabla}_{W}\breve{J}U=\breve{J}\breve{\nabla}_{W}U  \label{21a}
\end{equation}%
\cite{GH}.
\end{definition}

If $\breve{J}$ be a metallic structure, then (\ref{21}) is equivalent to 
\begin{equation}
\breve{g}(\breve{J}W,\breve{J}U)=p\breve{g}(\breve{J}W,U)+q\breve{g}(W,U),
\label{23}
\end{equation}%
for any $W,U\in \Gamma (T\breve{N}).$

\section{Radical Transversal Lightlike Submanifolds of Metallic
semi-Riemannian Manifolds}

In this section, we introduce radical transversal lightlike submanifolds of
a Metallic semi-Riemannian manifold.

\begin{definition}
Let $(\acute{N},g,S(T\acute{N}),S(T\acute{N}^{\perp }))$ be a lightlike
submanifold of a Metallic semi-Riemannian manifold $(\breve{N},\breve{g},%
\breve{J})$. If the following conditions are satisfied, then the lightlike
submanifold $\acute{N}$ is called a radical transversal lightlike
submanifold:%
\begin{eqnarray}
\breve{J}Rad\,T\acute{N} &=&ltr\,T\acute{N},  \label{24a} \\
\breve{J}S(T\acute{N}) &=&S(T\acute{N}).  \label{25a}
\end{eqnarray}
\end{definition}

\begin{proposition}
Let $\breve{N}$ be a Metallic semi-Riemannian manifold. In this case, there
is not 1-radical transversal lightlike submanifold of $\breve{N}.$
\end{proposition}

\begin{proof}
Let $\acute{N}$ be a 1-radical transversal lightlike submanifold. Hence, $%
Rad\,T\acute{N}=\{\xi \}$ and $ltr\,T\acute{N}=\{N\}.$ From the equation (%
\ref{23}), we have%
\begin{equation}
\breve{g}(\breve{J}\xi ,\xi )=\breve{g}(\xi ,\breve{J}\xi )=0.  \label{26b}
\end{equation}%
On the other hand, from (\ref{24a}), we have 
\begin{equation*}
\breve{g}(\breve{J}\xi ,\xi )=\breve{g}(\xi ,\breve{J}\xi )=\breve{g}(N,\xi
)=\breve{g}(\xi ,N)=1,
\end{equation*}%
which contradicts equation (\ref{26b}). The proof is completed.
\end{proof}

\begin{theorem}
Let $\acute{N}$ be a radical transversal lightlike submanifold of a Metallic
semi-Riemannian manifold $\breve{N}.$ In this case , the distribution $S(T%
\acute{N}^{\perp })$ is invariant with respect to $\breve{J}.$
\end{theorem}

\begin{proof}
For $V\in \Gamma (S(T\acute{N}^{\perp }))$ and $\xi \in \Gamma (Rad\,T\acute{%
N}),$ from (\ref{23}), we find 
\begin{equation*}
\breve{g}(\breve{J}V,\xi )=\breve{g}(V,\breve{J}\xi )=0,
\end{equation*}%
which implies that there is not a component in $ltr\,T\acute{N}$ of $\breve{J%
}V.$

Similarly, for $N\in \Gamma l(tr\,T\acute{N}),$\ from (\ref{23}), we have%
\begin{equation}
\breve{g}(\breve{J}V,N)=\breve{g}(V,\breve{J}N)=\frac{1}{p}\breve{g}(\breve{J%
}V,\breve{J}N).  \label{20a}
\end{equation}%
From definition of a radical transversal lightlike submanifold, for $\xi
_{1}\in \Gamma (Rad\,T\acute{N})$ and $N_{1}\in \Gamma (ltr\,T\acute{N}),$
we get%
\begin{equation*}
\breve{J}\xi _{1}=N_{1}.
\end{equation*}%
If we apply $\breve{J}$ to the last equation, we can write%
\begin{equation*}
p\breve{J}\xi _{1}+q\xi _{1}=\breve{J}N_{1},
\end{equation*}
which implies equation (\ref{20a}) equals zero. Namely, we see that, there
is no component of $\breve{J}V$ in $Rad\,T\acute{N}$ $.$

By a similar way, for $W\in \Gamma (S(T\acute{N})),$ we obtain 
\begin{equation*}
\breve{g}(\breve{J}V,W)=\breve{g}(V,\breve{J}W)=0,
\end{equation*}%
that is, there is no component of $\breve{J}V$\ in $S(T\acute{N}).$ Hence,
the proof is completed.
\end{proof}

\bigskip

\bigskip Let $\acute{N}$ be a radical transversal lightlike submanifold of
Metallic semi-Riemannian manifold $\breve{N}.$ $Q$ and $T$ denote projection
morphisms in $Rad\,T\acute{N}$ and $S(T\acute{N}),$ respectively. For any $%
W\in \Gamma (T\acute{N}),$ we can write%
\begin{equation}
W=TW+QW,  \label{28a}
\end{equation}%
where $TW\in \Gamma (S(T\acute{N}))$ and $QW\in \Gamma (Rad\,T\acute{N}).$
By applying $\breve{J}$ to (\ref{28a}), we have%
\begin{equation}
\breve{J}W=\breve{J}TW+\breve{J}QW.  \label{29b}
\end{equation}%
Here, if we write $\breve{J}TW=SW$ and $\breve{J}QW=LW,$ then (\ref{29b})
becomes 
\begin{equation}
\breve{J}W=SW+LW,  \label{30a}
\end{equation}%
where, $SW\in \Gamma (S(T\acute{N}))$ and $LW\in ltr\,T\acute{N}.$

Assume, $\acute{N}$ be a radical transversal submanifold of a locally
Metallic semi-Riemannian manifold $\breve{N}.$ From (\ref{21a}), (\ref{9})
and (\ref{11}), we have%
\begin{equation*}
\breve{\nabla}_{U}\left( SW+LW\right) {\small =}\breve{J}\left( \nabla
_{U}W+h^{l}(U,W\right) {\small +}h^{s}(U,W)),
\end{equation*}%
where $U,W\in \Gamma (T\acute{N}).$\ If we write $\breve{J}h^{l}(U,W)=K_{1}%
\breve{J}h^{l}(U,W)+K_{2}\breve{J}h^{l}(U,W),$ where $K_{1}$ and $K_{2}$ are
projection morphisms of $\breve{J}ltr\,T\acute{N}$ in $ltr\,T\acute{N}$ and $%
Rad\,T\acute{N}$, respectively$,$ we find 
\begin{equation*}
\left( 
\begin{array}{c}
\nabla _{U}SW+h\text{$^{l}$}(U,SW)+h^{s}(U,SW) \\ 
-A_{LW}U+\nabla _{U}^{l}LW+D^{s}(U,LW)%
\end{array}%
\right) =\left( 
\begin{array}{c}
S\nabla _{U}W+L\nabla _{U}W+\breve{J}h^{s}(U,W) \\ 
+K_{1}\breve{J}h^{l}(U,W)+K_{2}\breve{J}h^{l}(U,W)%
\end{array}%
\right) .
\end{equation*}%
Thus, by equating the tangent, screen transversal and lightlike transversal
parts components, we have%
\begin{eqnarray*}
\nabla _{U}SW-A_{LW}U &=&S\nabla _{U}W+K_{2}\breve{J}h^{l}(U,W), \\
h^{S}(U,SW)+D^{s}(U,LW) &=&\breve{J}h^{s}(U,W), \\
h^{l}(U,SW)+\nabla _{U}^{l}LW &=&L\nabla _{U}W+K_{1}\breve{J}h^{l}(U,W).
\end{eqnarray*}%
Therefore we give the following proposition.

\begin{proposition}
\bigskip Let $\acute{N}$ be a radical transversal lightlike submanifold of a
locally Metallic semi-Riemannian manifold $\breve{N}$. Then, we have 
\begin{eqnarray}
\left( \nabla _{U}S\right) W& =&A_{LW}U+K_{2}\breve{J}h^{l}(U,W),
\label{31a} \\
0 &=&h^{S}(U,SW)+D^{s}(U,LW)-\breve{J}h^{s}(U,W),  \label{32a} \\
0 &=&h^{l}(U,SW)+\nabla _{U}^{l}LW-L\nabla _{U}W-K_{1}\breve{J}h^{l}(U,W),
\label{33a}
\end{eqnarray}%
for $W,U\in \Gamma (T\acute{N}).$
\end{proposition}

\begin{theorem}
Let $\acute{N}$ be a radical transversal lightlike submanifold of a locally
Metallic semi-Riemannian manifold $\breve{N}$. Then, the induced connection $%
\nabla $ on $\acute{N}$ is a metric connection if and only if there is no
component of $A_{\breve{J}\xi }W$ in $\Gamma (S(T\acute{N})),$ for $W\in
\Gamma (T\acute{N})$ and $\xi \in \Gamma (Rad\,T\acute{N}).$
\end{theorem}

\begin{proof}
Assume that the induced connection $\nabla $ is a metric connection. In this
case, for $W\in \Gamma (T\acute{N})$ and $\xi \in \Gamma (Rad\,T\acute{N}),$ 
$\nabla _{\xi }W\in \Gamma (Rad\,T\acute{N}).$ Here, $U\in \Gamma (S(T\acute{%
N})),$ we have 
\begin{equation*}
g\left( \nabla _{\xi }W,U\right) =\breve{g}(\breve{\nabla}_{\xi }W,U)=0.
\end{equation*}%
If we use (\ref{23}), we find%
\begin{equation*}
0=\breve{g}(\breve{J}\breve{\nabla}_{\xi }W,\breve{J}U)-p\breve{g}(\breve{%
\nabla}_{\xi }W,\breve{J}U),
\end{equation*}%
and from (\ref{10}), we have%
\begin{equation*}
\breve{g}(A_{\breve{J}\xi }W,\breve{J}U)=0,
\end{equation*}%
which implies, there is no component of $A_{\breve{J}\xi }W$ in $\Gamma (S(T%
\acute{N})).$

Since the converse is obvious, then we omit it.
\end{proof}

Now, we shall investigate the conditions for integrability of the
distributions involved in the definition of radical transversal lightlike
submanifolds.

\begin{theorem}
Let $\acute{N}$ be a radical transversal lightlike submanifold of a locally
Metallic semi-Riemannian manifold $\breve{N}$. In this case, the screen
distribution\ is integrable if and only if 
\begin{equation*}
h^{l}(U,SW)=h^{l}(W,SU),
\end{equation*}%
for $W,U\in \Gamma (S(T\acute{N})).$
\end{theorem}

\begin{proof}
For $W,U\in \Gamma (S(T\acute{N})),$ if we use equation (\ref{33a}) and by
interchanging the roles of $U$ and $W,$ we find 
\begin{equation*}
h^{l}(U,SW)-h^{l}(W,SU)-K_{1}\left( \breve{J}h^{l}(U,W)-\breve{J}%
h^{l}(W,U)\right) =L\left[ U,W\right] .
\end{equation*}%
Since $h^{l}$ is symmetric, we obtain%
\begin{equation*}
h^{l}(U,SW)-h^{l}(W,SU)=L\left[ U,W\right] ,
\end{equation*}%
which completes the proof.
\end{proof}

\begin{theorem}
\bigskip Let $\acute{N}$ be a radical transversal lightlike submanifold of a
locally Metallic semi-Riemannian manifold $\breve{N}$. The radical
distribution is integrable if and only if 
\begin{equation*}
A_{LU}W=A_{LW}U,
\end{equation*}%
for $V,W\in \Gamma (Rad\,T\acute{N}).$
\end{theorem}

\begin{proof}
For $V,W\in \Gamma (Rad\,T\acute{N}),$ if we use equation (\ref{31a}), we
have 
\begin{equation*}
-S\nabla _{U}W=A_{LW}U+K_{2}\breve{J}h^{l}(U,W),
\end{equation*}%
by virtue of $SW=0.$ By changing the roles of $U$ and $W,$ we find%
\begin{equation*}
S\left( \nabla _{W}U-\nabla _{U}W\right) =A_{LU}W-A_{LW}U+K_{2}\left( \breve{%
J}h^{l}(W,U)-\breve{J}h^{l}(U,W)\right) .
\end{equation*}%
Since $h^{l}$ is known to be symmetric, we obtain 
\begin{equation*}
S\left[ W,U\right] =A_{LU}W-A_{LW}U.
\end{equation*}%
Therefore, the proof is completed.
\end{proof}

\begin{theorem}
Let $\acute{N}$ be a radical transversal lightlike submanifold of a locally
Metallic semi-Riemannian manifold $\breve{N}$. Then the radical distribution
define as a totally geodesic foliation if and only if 
\begin{equation*}
h^{\ast }(W,\breve{J}Z)=ph^{\ast }(W,Z),
\end{equation*}%
for $W\in \Gamma (Rad\,T\acute{N}),$ $Z\in \Gamma (S(T\acute{N})).$
\end{theorem}

\begin{proof}
By using the definition of a lightlike submanifold, it is known that the
radical distribution defines totally geodesic foliation if and only if 
\begin{equation*}
\breve{g}\left( \nabla _{W}U,Z\right) =0,
\end{equation*}%
for $W,U\in \Gamma (Rad\,T\acute{N})$ and $Z\in S(T\acute{N}).$ Since $%
\breve{\nabla}$ is a metric connection, if we use (\ref{9}), (\ref{21a}) and
(\ref{23}), we have 
\begin{equation*}
\breve{g}(\breve{J}U,\breve{\nabla}_{W}\breve{J}Z)-p\breve{g}(\breve{J}U,%
\breve{\nabla}_{W}Z)=0.
\end{equation*}%
Then from (\ref{14}), we get 
\begin{equation*}
\breve{g}\left( \breve{J}U,h^{\ast }(W,\breve{J}Z)-ph^{\ast }(W,Z)\right) =0.
\end{equation*}%
Hence, the proof is completed.
\end{proof}

\begin{theorem}
Let $\acute{N}$ be a radical transversal lightlike submanifold of a locally
Metallic semi-Riemannian manifold $\breve{N}$. Then the screen distribution
defines a totally geodesic foliation if and only if either 
\begin{equation*}
h^{\ast }(W,\breve{J}U)+K_{2}h^{l}(W,\breve{J}U)=p(h^{\ast
}(W,U)+K_{2}h^{l}(W,U)),
\end{equation*}%
or there is no component of $\breve{J}N$ in $ltr\,T\acute{N}$ for $W,U\in
\Gamma (S(T\acute{N})),$ $N\in \Gamma (ltr\,T\acute{N}).$
\end{theorem}

\begin{proof}
Since the screen distribution defines a totally geodesic foliation if and
only if 
\begin{equation*}
\breve{g}\left( \nabla _{W}U,N\right) =0,
\end{equation*}%
for any $W,U\in \Gamma (S(T\acute{N})),$ $N\in \Gamma (ltr\,T\acute{N}).$
Here, if we use (\ref{9}), then we have%
\begin{equation*}
\breve{g}(\breve{\nabla}_{W}U,N)=0.
\end{equation*}%
Also from (\ref{23}) and (\ref{21a}), we have%
\begin{equation*}
\breve{g}(\breve{\nabla}_{W}\breve{J}U,\breve{J}N)-p\breve{g}(\breve{\nabla}%
_{W}U,\breve{J}N)=0.
\end{equation*}%
By using (\ref{9}) and (\ref{14}) in the last equation, we find 
\begin{equation*}
\breve{g}(h^{\ast }(W,\breve{J}U)+K_{2}h^{l}(W,\breve{J}U),\breve{J}N)-p%
\breve{g}(h^{\ast }(W,U)+K_{2}h^{l}(W,U),\breve{J}N)=0.
\end{equation*}%
Therefore, we conclude.
\end{proof}

\section{Transversal Lightlike Submanifolds of Metallic semi-Riemannian
Manifolds}

In this section, we give definition of transversal lightlike submanifolds
and investigate the geometry of distributions.

\begin{definition}
\label{def-tr}Let ($\acute{N},g,S(T\acute{N}),S(T\acute{N}^{\perp }))$ be a
lightlike submanifold of a Metallic semi-Riemannian manifold $(\breve{N},%
\breve{g},\breve{J})$. If the following conditions are satisfied, then the
lightlike submanifold $\acute{N}$ is called transversal lightlike
submanifold:%
\begin{eqnarray*}
\breve{J}Rad\,T\acute{N} &=&ltr\,T\acute{N}, \\
\breve{J}(S(T\acute{N})) &\subseteq &S(T\acute{N}^{\perp }).
\end{eqnarray*}
\end{definition}

We shall denote the orthogonal complement sub bundle to $\breve{J}(S(T\acute{%
N}))$ in $S(T\acute{N}^{\perp })$ by $\mu .$

\begin{proposition}
Let $\acute{N}$ be a transversal lightlike submanifold of a locally Metallic
semi-Riemannian manifold $\breve{N}$. In this case, the distribution $\mu $
is invariant according to $\breve{J}.$
\end{proposition}

\begin{proof}
For $V\in \Gamma \left( \mu \right) ,$ $\xi \in \Gamma (Rad\,T\acute{N})$
and $N\in \Gamma (ltr\,T\acute{N}),$ from (\ref{20}), (\ref{21}) and (\ref%
{23}), we have 
\begin{equation}
\breve{g}(\breve{J}V,\xi )=\breve{g}(V,\breve{J}\xi )=0,  \label{34a}
\end{equation}%
and 
\begin{equation}
\breve{g}(\breve{J}V,N)=\breve{g}(V,\breve{J}N)=0.  \label{35a}
\end{equation}%
Therefore, there is no component of $\breve{J}V$ in $Rad\,T\acute{N}$ and $%
ltr\,T\acute{N}.$

Similarly, for $W\in \Gamma (S(T\acute{N}))$ and $V_{1}\in \Gamma (S(T\acute{%
N}^{\perp })),$ we have 
\begin{equation}
\breve{g}(\breve{J}V,W)=\breve{g}(V,\breve{J}W)=0,  \label{36a}
\end{equation}%
and 
\begin{equation}
\breve{g}(\breve{J}V,V_{1})=\breve{g}(V,\breve{J}V_{1})=0,  \label{37a}
\end{equation}%
which imply that there is no component of $\breve{J}V$ in $S(T\acute{N})$
and $\breve{J}(S(T\acute{N})).$ From (\ref{34a}), (\ref{35a}), (\ref{36a})
and (\ref{37a}), we conclude.
\end{proof}

\begin{proposition}
\label{prop-tr}There does not exist a $1$-lightlike transversal lightlike
submanifold of a locally Metallic semi-Riemannian manifold.
\end{proposition}

\begin{proof}
Assume that $\acute{N}$ is a $1$-lightlike transversal lightlike submanifold
of a locally Metallic semi-Riemannian manifold $\breve{N}$. In this case, $%
Rad\,T\acute{N}=Sp\left\{ \xi \right\} $ and $ltr\,T\acute{N}=Sp\left\{
N\right\} .$ From (\ref{20}) and (\ref{23}), we obtain 
\begin{equation}
\breve{g}(\breve{J}\xi ,\xi )=\breve{g}(\xi ,\breve{J}\xi )=0.  \label{38a}
\end{equation}%
On the other hand, from the fact that $\breve{J}Rad\,T\acute{N}=ltr\,T\acute{%
N}$ , we have $\breve{J}\xi =N\in \Gamma (ltr\,T\acute{N}).$ So, we find%
\begin{equation*}
\breve{g}(\xi ,\breve{J}\xi )=\breve{g}(\xi ,N)=1,
\end{equation*}%
which contradicts with (\ref{38a}). The proof is completed.
\end{proof}

From Definition \ref{def-tr} and Proposition \ref{prop-tr}, we have

\begin{corollary}
Let $\acute{N}$ be a transversal lightlike submanifold of a locally Metallic
semi-Riemannian manifold $\breve{N}$. Then,

\begin{description}
\item[(i)] $\dim (Rad\,T\acute{N})\geq 2,$

\item[(ii)] The transversal lightlike submanifold of $3$-dimensional is $2$%
-lightlike.
\end{description}
\end{corollary}

\bigskip

Let $\acute{N}$ be a transversal lightlike submanifold of a locally Metallic
semi-Riemannian manifold $\breve{N}$. $Q$ and $T$ are projection morphisms
in $Rad\,T\acute{N}$ and $S(T\acute{N}),$ respectively. For any $W\in \Gamma
(T\acute{N}),$ we can write%
\begin{equation}
W=TW+QW,  \label{39a}
\end{equation}%
where, $TW\in \Gamma (S(T\acute{N}))$ and $QW\in \Gamma (Rad\,T\acute{N}).$
If we applied $\breve{J}$ to (\ref{39a})$,$ we have%
\begin{equation}
\breve{J}W=\breve{J}TW+\breve{J}QW.  \label{40a}
\end{equation}%
By writing $\breve{J}TW=KW$ and $\breve{J}QW=LW,$ the expression (\ref{40a})
is%
\begin{equation}
\breve{J}W=KW+LW.  \label{41a}
\end{equation}%
Here, $KW\in \Gamma (S(T\acute{N}^{\perp }))$ and $LW\in ltr\,T\acute{N}.$
Besides, let $D$ and $E$ be projection morphisms in $\breve{J}S(T\acute{N})$
and $\mu $ in $S(T\acute{N}^{\perp })$, respectively. For $V\in \Gamma (S(T%
\acute{N}^{\perp })),$ we write%
\begin{equation}
V=DV+EV.  \label{42a}
\end{equation}%
By applying $\breve{J}$ \ to (\ref{42a}), we have%
\begin{equation}
\breve{J}V=\breve{J}DV+\breve{J}EV.  \label{43a}
\end{equation}%
If we write $\breve{J}DV=BV$ and $\breve{J}EV=CV,$ expression (\ref{43a})
becomes%
\begin{equation}
\breve{J}V=BV+CV,  \label{43b}
\end{equation}%
where $BV\in \breve{J}S(T\acute{N})\oplus S(T\acute{N}),$ $CV\in \Gamma
\left( \mu \right) .$ Since $\breve{N}$ is a locally Metallic
semi-Riemannian manifold, then from (\ref{9}), (\ref{11})and (\ref{41a}), we
have%
\begin{equation}
\left( 
\begin{array}{c}
-A_{KW}U+\nabla _{U}^{s}KW+D^{l}(U,KW) \\ 
-A_{LW}U+\nabla _{U}^{l}LW+D^{s}(U,LW)%
\end{array}%
\right) =\left( 
\begin{array}{c}
K\nabla _{U}W+L\nabla _{U}W+\breve{J}h^{l}(U,W) \\ 
+Bh^{s}(U,W)+Ch^{s}(U,W)%
\end{array}%
\right) ,  \label{44a}
\end{equation}%
where $U,W\in \Gamma (T\acute{N}).$ For projection morphisms $K_{1}$ and $%
K_{2}$ of $\breve{J}ltr\,T\acute{N}$ in $ltr\,T\acute{N}$ and $Rad\,T\acute{N%
}$\ respectively, we write%
\begin{equation*}
\breve{J}h^{l}(U,W)=K_{1}\breve{J}h^{l}(U,W)+K_{2}\breve{J}h^{l}(U,W).
\end{equation*}%
Also, for projection morphisms $S_{1}$ and $S_{2}$ of $\breve{J}S(T\acute{N}%
^{\perp })$ in $\breve{J}S(T\acute{N})\subseteq S(T\acute{N}^{\perp })$ and $%
S(T\acute{N}),$we have%
\begin{equation*}
Bh^{s}(U,W)=S_{1}Bh^{s}(U,W)+S_{1}Bh^{s}(U,W).
\end{equation*}%
Therefore (\ref{44a}) can be rewritten as 
\begin{equation*}
\left( 
\begin{array}{c}
-A_{KW}U+\nabla _{U}^{s}KW+D^{l}(U,KW) \\ 
-A_{LW}U+\nabla _{U}^{l}LW+D^{s}(U,LW)%
\end{array}%
\right) =\left( 
\begin{array}{c}
K\nabla _{U}W+L\nabla _{U}W+K_{1}\breve{J}h^{l}(U,W) \\ 
+K_{2}\breve{J}h^{l}(U,W)+S_{1}Bh^{s}(U,W) \\ 
+S_{2}Bh^{s}(U,W)+Ch^{s}(U,W)%
\end{array}%
\right) .
\end{equation*}%
If we equate the tangent and transversal parts of the above equation, then
we get 
\begin{eqnarray}
-A_{KW}U-A_{LW}U &=&K_{2}\breve{J}h^{l}(U,W)+S_{2}Bh^{s}(U,W){\small ,}
\label{45a} \\
\nabla _{U}^{s}KW+D^{s}(U,LW) &=&K\nabla _{U}W+S_{1}Bh^{s}(U,W)+Ch^{s}(U,W)%
{\small ,}  \label{46a} \\
D^{l}(U,KW)+\nabla _{U}^{l}LW &=&L\nabla _{U}W+K_{1}\breve{J}h^{l}(U,W)%
{\small .}  \label{47a}
\end{eqnarray}

Now we shall investigate the integrable of the distributions on transversal
lightlike submanifolds.

\begin{theorem}
Let $\acute{N}$ be a transversal lightlike submanifold of a locally Metallic
semi-Riemannian manifold $\breve{N}$. Then the radical distribution is
integrable if and only if 
\begin{equation*}
D^{s}(U,LW)=D^{s}(W,LU),
\end{equation*}%
for $V,W\in \Gamma (Rad\,T\acute{N}).$
\end{theorem}

\begin{proof}
For $V,W\in \Gamma (Rad\,T\acute{N}),$ from equation (\ref{46a}), by
interchanging roles of $W$ and $U$, we find%
\begin{equation*}
\nabla _{U}^{s}KW-\nabla _{W}^{s}KU+D^{s}(U,LW)-D^{s}(W,LU)-K\left( \nabla
_{U}W-\nabla _{W}U\right) =0,
\end{equation*}%
since $h^{s}$ is symmetric. Also, we have $\nabla _{U}^{s}KW=\nabla
_{W}^{s}KU=0.$ Then, we get%
\begin{equation*}
D^{s}(U,LW)-D^{s}(W,LU)=K\left[ U,W\right] ,
\end{equation*}%
which completes the proof .
\end{proof}

\begin{theorem}
Let $\acute{N}$ be a\ transversal lightlike submanifold of a locally
Metallic semi-Riemannian manifold $\breve{N}$. Then the screen distribution
\ is integrable if and only if 
\begin{equation*}
D^{l}(U,KW)=D^{l}(W,KU),
\end{equation*}%
$W,U\in \Gamma (S(T\acute{N})).$
\end{theorem}

\begin{proof}
From (\ref{47a}), the fact that $h^{l}$ is symmetric and $LW=LU=0$, we have%
\begin{equation*}
D^{l}(U,KW)-D^{l}(W,KU)=L\left[ U,W\right] ,
\end{equation*}%
by interchanging the roles of $W,U\in \Gamma (S(T\acute{N})).$ Thus, the
proof is completed.
\end{proof}

\begin{theorem}
Let $\acute{N}$ be a\ transversal lightlike submanifold of a locally
metallic semi-Riemannian manifold $\breve{N}$. Then the screen distribution
defines a totally geodesic foliation if and only if $D^{l}(W,\breve{J}%
U)=-ph^{l}(W,U)$, $h^{\ast }(W,U)=0$ and there is no component of $A_{\breve{%
J}U}W$ in $Rad\,T\acute{N},$ for $W,U\in \Gamma (S(T\acute{N}))$, $N\in
\Gamma (ltr\,T\acute{N}).$
\end{theorem}

\begin{proof}
By the definition of a lightlike submanifold, it is known that $S(T\acute{N})
$ defines a totally geodesic foliation if and only if 
\begin{equation*}
\breve{g}\left( \nabla _{W}U,N\right) =0,
\end{equation*}%
where $W,U\in \Gamma (S(T\acute{N}))$ and $N\in \Gamma (ltr\,T\acute{N}).$If
we use (\ref{21}), (\ref{21a}) and (\ref{23}), we find%
\begin{equation*}
0=\breve{g}(\breve{\nabla}_{W}\breve{J}U,\breve{J}N)-p\breve{g}(\breve{\nabla%
}_{W}U,\breve{J}N).
\end{equation*}%
Since $\breve{J}U\in \Gamma (S(T\acute{N}))$, from equation (\ref{9}) and (%
\ref{11}), we have%
\begin{equation*}
\breve{g}(-A_{\breve{J}U}W+D^{l}(W,\breve{J}U),\breve{J}N)-p\breve{g}(\nabla
_{W}U+h^{l}(W,U),\breve{J}N)=0.
\end{equation*}%
Then by using (\ref{14}), we obtain%
\begin{equation*}
\breve{g}(-A_{\breve{J}U}W+D^{l}(W,\breve{J}U)-ph^{\ast }(W,U)+ph^{l}(W,U),%
\breve{J}N)=0,
\end{equation*}%
which completes the proof.
\end{proof}

\begin{theorem}
Let $\acute{N}$ be a transversal lightlike submanifold of a locally metallic
semi-Riemannian manifold $\breve{N}$. Then the radical distribution defines
a totally geodesic foliation if and only if there is no component in $Rad\,T%
\acute{N}$ of $A_{\breve{J}Z}W,$ that is, either $K_{2}\breve{J}h^{l}(W,Z)=0$
or $-A_{\breve{J}Z}W=S_{2}Bh^{s}(W,Z),$ for $W,U\in \Gamma (Rad\,T\acute{N})$%
, $Z\in \Gamma (S(T\acute{N})).$
\end{theorem}

\begin{proof}
The radical distribution defines a totally geodesic foliation if and only if 
\begin{equation*}
\breve{g}\left( \nabla _{W}U,Z\right) =0,
\end{equation*}%
for $W,U\in \Gamma (Rad\,T\acute{N})$ and $Z\in S(T\acute{N}).$ From (\ref{9}%
), we find 
\begin{equation*}
\breve{g}\left( \nabla _{W}U,Z\right) =\breve{g}(\breve{\nabla}_{W}U,Z)=0.
\end{equation*}%
Since, $\breve{\nabla}$ is a metric connection, from (\ref{21}), (\ref{21a})
and (\ref{23}), we have%
\begin{equation*}
0=-\breve{g}(\breve{J}U,\breve{\nabla}_{W}\breve{J}Z)+p\breve{g}(\breve{J}U,%
\breve{\nabla}_{W}Z).
\end{equation*}%
For $\breve{J}Z\in \Gamma (S(T\acute{N}^{\perp })),$ from (\ref{11}) and (%
\ref{14}), we get%
\begin{equation*}
0=\breve{g}(\breve{J}U,A_{\breve{J}Z}W)+p\breve{g}(\breve{J}U,h^{\ast
}(W,Z)).
\end{equation*}%
Here, since $\breve{J}U\in \Gamma (ltr\,T\acute{N}),$ we conclude that
either there is no component of $A_{\breve{J}Z}W$ in $Rad\,T\acute{N}$ or by
changing the roles of $U$, $W$ and taking $U=Z,$ we have $K_{2}\breve{J}%
h^{l}(W,Z)=0,$ by virtue of%
\begin{equation*}
-A_{\breve{J}Z}W=K_{2}\breve{J}h^{l}(W,Z)+S_{2}Bh^{s}(W,Z).
\end{equation*}
\end{proof}

\begin{theorem}
Let $\acute{N}$ be a transversal lightlike submanifold of a locally Metallic
semi-Riemannian manifold $\breve{N}$. Then the induced connection on $\acute{%
N}$ is a metric connection if and only if 
\begin{equation*}
Q_{1}\breve{J}D^{s}(W,\breve{J}\xi )=pM_{1}\breve{J}h^{s}(W,\xi ),
\end{equation*}%
for $W,U\in \Gamma (T\acute{N}),\xi \in \Gamma (Rad\,T\acute{N}).$
\end{theorem}

\begin{proof}
For $W,U\in \Gamma (T\acute{N})$ and $\xi \in \Gamma (Rad\,T\acute{N}),$we
have%
\begin{equation*}
\breve{\nabla}_{W}\breve{J}\xi =\breve{J}\breve{\nabla}_{W}\xi .
\end{equation*}%
From equations (\ref{9}) and (\ref{10}), we write%
\begin{equation*}
-A_{\breve{J}\xi }W+\nabla _{W}^{l}\breve{J}\xi +D^{s}(W,\breve{J}\xi )=%
\breve{J}\left( \nabla _{W}\xi +h^{l}(W,\xi )+h^{s}(W,\xi )\right) .
\end{equation*}%
If we apply $\breve{J}$ to the above equation and use (\ref{20}), (\ref{41a}%
) and (\ref{43b}) we obtain 
\begin{equation}
\text{$\left( 
\begin{array}{c}
-KA_{\breve{J}\xi }W-LA_{\breve{J}\xi }W \\ 
+T_{1}\breve{J}\nabla _{W}^{l}\breve{J}\xi +T_{2}\breve{J}\nabla _{W}^{l}%
\breve{J}\xi  \\ 
+Q_{1}\breve{J}D^{s}\left( W,\breve{J}\xi \right) +Q_{2}\breve{J}D^{s}\left(
W,\breve{J}\xi \right) 
\end{array}%
\right) $=$\left( 
\begin{array}{c}
p\breve{J}\nabla _{W}\xi +q\nabla _{W}\xi + \\ 
p\breve{J}h^{l}(W,\xi )+qh^{l}(W,\xi ) \\ 
+p\breve{J}h^{s}(W,\xi )+qh^{s}(W,\xi )%
\end{array}%
\right) ,$}  \label{48b}
\end{equation}%
for $V\in $ $\Gamma (S(T\acute{N}^{\perp })),$ where $T_{1}$ and $T_{2}$ are
projection morphisms of $\breve{J}\nabla _{W}^{l}\breve{J}\xi $ in $Rad\,T%
\acute{N}$ and $ltr\,T\acute{N},$ respectively. Then we have 
\begin{equation*}
\breve{J}\nabla _{W}^{l}\breve{J}\xi =T_{1}\breve{J}\nabla _{W}^{l}\breve{J}%
\xi +T_{2}\breve{J}\nabla _{W}^{l}\breve{J}\xi .
\end{equation*}%
Also, for projection morphisms $M_{1}$ and $M_{2}$ are of $\breve{J}%
h^{s}(W,\xi )$ in $S(T\acute{N})$ and $\breve{J}S(T\acute{N}),$
respectively, then we have%
\begin{equation*}
\breve{J}h^{s}(W,\xi )=M_{1}\breve{J}h^{s}(W,\xi )+M_{2}\breve{J}h^{s}(W,\xi
).
\end{equation*}%
Additionally, we get 
\begin{equation*}
\breve{J}D^{s}(W,\breve{J}\xi )=Q_{1}\breve{J}D^{s}(W,\breve{J}\xi )+Q_{2}%
\breve{J}D^{s}(W,\breve{J}\xi ),
\end{equation*}%
where $Q_{1}$ and $Q_{2}$ are projection morphisms of $\breve{J}D^{s}(W,%
\breve{J}\xi )$ in $S(T\acute{N})$ and $S(T\acute{N}^{\perp }),$
respectively. By equating tangent parts in equation (\ref{48b}), we find 
\begin{equation*}
\text{$\frac{1}{q}$}\left( T_{1}\breve{J}\nabla _{W}^{l}\breve{J}\xi +Q_{1}%
\breve{J}D^{s}(W,\breve{J}\xi )-pM_{1}\breve{J}h^{s}(W,\xi )-pK_{2}\breve{J}%
h^{l}(W,\xi )\right) =\nabla _{W}\xi .
\end{equation*}%
Therefore $\nabla _{W}\xi $ is belong to $Rad\,T\acute{N}$ if and only if 
\begin{equation*}
Q_{1}\breve{J}D^{s}(W,\breve{J}\xi )=pM_{1}\breve{J}h^{s}(W,\xi ).
\end{equation*}%
This completes the proof.
\end{proof}

\begin{example}
Let $(\breve{N}=\mathbb{R}_{2}^{5},\breve{g},\breve{J})$ be the $5$%
-dimensional semi-Euclidean space with the semi-Euclidean metric of
signature $(-,+,-,+,+)$ and the structure $\breve{J}$ given by%
\begin{equation*}
\breve{J}(x_{1},x_{2},x_{3},x_{4},x_{5})=(\left( p-\sigma \right)
x_{1},\sigma x_{2},\left( p-\sigma \right) x_{3},\sigma x_{4},\sigma x_{5}),
\end{equation*}
where $(x_{1},x_{2},x_{3},x_{4},x_{5})$ is the standard coordinate system of 
$\mathbb{R}_{2}^{5}$ . If we take $\sigma =\frac{p+\sqrt{p^{2}+4q}}{2},$
then we have 
\begin{equation*}
\breve{J}^{2}=p\breve{J}+qI,
\end{equation*}%
which implies $\breve{J}$ is a metallic structure on $\mathbb{R}_{2}^{5}.$
Hence,$(\breve{N}=\mathbb{R}_{2}^{5},\breve{g},\breve{J})$ is a Metallic
semi-Riemannian manifold. Let $\acute{N}$ be a submanifold in $\breve{N}$
defined by%
\begin{equation*}
x_{2}=0,\quad x_{4}=\sigma x_{1}+\sigma x_{3}
\end{equation*}%
Then, $T\acute{N}=Sp\{W_{1},W_{2},W_{3}\},$ where%
\begin{equation*}
W_{1}=\frac{\partial }{\partial x_{1}}+\sigma \frac{\partial }{\partial x_{4}%
},\quad W_{2}=\frac{\partial }{\partial x_{3}}+\sigma \frac{\partial }{%
\partial x_{4}},\text{\quad }W_{3}=\frac{\partial }{\partial x_{5}}.
\end{equation*}%
It is easy to check that $\acute{N}$ is a lightlike submanifold. Therefore,%
\begin{eqnarray*}
Rad\,T\acute{N} &=&Sp\{\xi =\sigma W_{1}-\sigma W_{2}+\sigma \sqrt{2}W_{3}\},
\\
ltr\,T\acute{N} &=&Sp\left\{ N=\frac{1}{2\sigma ^{2}\left( 2\sigma -p\right) 
}\left( \left( p-\sigma \right) \frac{\partial }{\partial x_{1}}-\left(
p-\sigma \right) \frac{\partial }{\partial x_{3}}+\sigma \sqrt{2}\frac{%
\partial }{\partial x_{5}}\right) \right\} , \\
S(T\acute{N}) &=&Sp\{W_{3}\},
\end{eqnarray*}%
and we have%
\begin{equation*}
N=\frac{1}{4\sigma ^{3}}\breve{J}\xi ,
\end{equation*}%
for $p=0.$ That is, $\breve{J}\xi \in \Gamma (ltr\,T\acute{N})$ and $\breve{J%
}W_{3}=\sigma W_{3}\in S(T\acute{N})$. Therefore, $\acute{N}$ is a radical
transversal lightlike submanifold of $(\breve{N}=\mathbb{R}_{2}^{5},\breve{g}%
,\breve{J})$.
\end{example}

\end{document}